\documentclass[a4paper,11pt]{amsart} 

\usepackage{amsmath,amssymb}
\usepackage{mathrsfs,bm}
\usepackage{comment}
\usepackage{color}
\usepackage{enumerate}
\usepackage[dvipdfm,
colorlinks=true,
linkcolor=blue,
filecolor=blue,
urlcolor=blue
]{hyperref}

\newtheorem{thm} {Theorem}  [section]
\newtheorem{lem} {Lemma} [section]
\newtheorem{defn} {Definition}  [section]
\newtheorem{prop} {Proposition} [section]
\newtheorem{cor}  {Corollary}[section]
\newtheorem{rem}  {Remark}  [section]

\newtheorem{example}  {Example}  [section]
\newtheorem{fact}{Fact}[section]

\newcommand{\dis}{\displaystyle}
\newcommand{\N}{\mathbb N}	
\newcommand{\Z}{\mathbb Z}	
\newcommand{\Q}{\mathbb Q}	
\newcommand{\R}{\mathbb R}	
\newcommand{\1}{\mathbf 1}	

\allowdisplaybreaks[1]

\begin{document}
\numberwithin{equation}{section}


\title[Convergence and projection Markov property]{Convergences and projection Markov property of Markov processes on ultrametric spaces}
	\author{Kohei SUZUKI}
	\address{Department of Mathematics, Kyoto University, Kitashirakawa, Sakyo-ku, Kyoto-Fu, 606-8224, Japan.}
	
	\maketitle

\begin{abstract}
Let $(S,\rho)$ be an ultrametric space with certain conditions and $S^k$ be the quotient space of $S$ with respect to the partition by balls with a fixed radius $\phi(k)$.
We prove that, for a Hunt process $X$ on $S$ associated with a Dirichlet form $(\mathcal E, \mathcal F)$, a Hunt process $X^k$ on $S^k$ associated with the averaged Dirichlet form $(\mathcal E^k, \mathcal F^k)$
is Mosco convergent to $X$, and under certain additional conditions, $X^k$ converges weakly to $X$.
Moreover, we give a sufficient condition for the Markov property of $X$ to be preserved under the canonical projection $\pi^k$ to $S^k$.
In this case, we see that the projected process $\pi^k\circ X$ is identical in law to $X^k$ and  converges almost surely to $X$. 
\end{abstract}

\section{Introduction}
A metric space $(S,\rho)$ is called an {\it ultrametric space} if the metric $\rho$ satisfies the following inequality:
	\begin{align}
		\rho(x,z) \le \max\{\rho(x,y),\rho(y,z)\} \quad (\forall x,y,z\in S), \label{ineq: STI}
	\end{align}
which is obviously stronger than the usual triangle inequality.
In this paper, we always assume the following conditions:
		\begin{enumerate}[(U.1)]
		\setcounter{enumi}{0}
			\item \label{cond: U.0} $(S,\rho)$ is a locally compact complete ultrametric space. 
			\item \label{cond: U.1} Any closed ball is compact. 
			\item  \label{cond: U.2}There exist an integer valued function $r: S \times S \to \Z\cup\{\infty\}$ and a strictly decreasing function $\phi : \Z\cup\{\infty\} \to \R$ with $\phi(\infty) =0$  such that 
						\begin{align}
							\rho(x,y) = \phi(r(x,y)) \quad (\forall x,y\in S). \notag
						\end{align}
			\item \label{cond: U.3} A measure $\mu$ is a Radon measure on $S$ assigning strictly positive finite values to all closed balls of positive radius.
		\end{enumerate}
Remark that separability of  $S$ follows from the above conditions. We denote $B_x^k:=\{y \in S: \rho(x,y) \le \phi(k)\}$.

Ultrametric spaces have many important examples in various fields of mathematics. 
One of the best known examples is the field $\Q_p$ of $p$-adic numbers equipped with the metric $\rho(x,y)=\|x-y\|_p$ where $\|\cdot\|_p$ is the $p$-adic norm. The field $\Q_p$
 is originated in Number theory and now is investigated in various fields.
 A lot of interesting  studies of Markov processes on $\Q_p$(or more generally, on local fields) have been investigated by Albeverio, Kaneko, Karwowski, Kochubei, Yasuda, Zhao and other authors 
(See \cite{Albeverio:1994hc, Albeverio:2002fma, Albeverio:1999uv, Albeverio:2001us, Kaneko:2007uh, Kochubei:1997wo, Yasuda:1996wi, Yasuda:2006wr} and references therein).
Recently in Karwowski--Yasuda \cite{Karwowski:2010it}, they studied an application of Markov processes on $\Q_p$ to {\it spin glasses}, which is an important subject in the field of physics.
The study of Markov processes on $\Q_p$ is important both from the mathematical viewpoint and the physical viewpoint.
In the above studies, they essentially used not only the ultrametric structure of $\Q_p$, but also algebraic structures, such as rings  or topological groups.
However, recently several authors studied Markov processes on more general ultrametric spaces without any algebraic structures,  such as the endpoints of locally-finite trees (called {\it leaves of multibranching trees} in 
Albeverio--Karwowski \cite{Albeverio:2008uk}, and {\it non-compact Cantor set} in Kigami \cite{KIGAMI:wb}). For more details, see Albeverio--Karwowski \cite{Albeverio:2008uk}, Kigami \cite{KIGAMI:2010vm, KIGAMI:wb}, 
Bendikov--Grigor'yan--Pittet \cite{Bendikov:2012ud}, Woess 	\cite{Woess:2012tk} and Bendikov--Grigr'yan--Pittet--Woess \cite{Bendikov:2013te} and references therein.
In this paper,  we do not assume any algebraic structures.

Ultrametric spaces have a remarkable geometrical property, which is quite different from the usual Euclidean spaces,  that 
any two balls with the common radius are either disjoint or identical.
From this fact and the conditions (U.\ref{cond: U.0})--(U.\ref{cond: U.2}), it follows that the family of balls with radius $\phi(k)$ forms a countable partition of the whole space $S$.
Let $S^k$ denote the quotient space with respect to this partition.
For later arguments, we embed $S^k$ to $S$ through an arbitrarily fixed $I^k$ satisfying $I^k([x]) \in B_x^k$ where $[x]$ denotes the equivalence class containing $x$.
Note that later arguments do not depend on a particular choice of $I^k$. 
Then the following question arises:
	\begin{description}
		\item[(Q1)] Can we approximate a Markov process on $S$ by Markov chains on $S^k$?
	\end{description}
Since ultrametric spaces are totally disconnected, we can only consider pure jump processes on $S$.
Let us consider the following bilinear form $(\mathcal E,\mathcal F)$ on $L^2(S;\mu)$:
		\begin{align*}
				&\mathcal E (u,v)= \frac{1}{2}\int_{S \times S \setminus d} ( u(x) - u(y) )(v(x)-v(y)) J(x,y) \mu(dx)\mu(dy),
		\end{align*}
				with its domain $\mathcal F = \overline{D_0}^{\mathcal E_1}$
		for a non-negative Borel measurable function  $J(x,y)$ on $ S\times S \setminus d $.
		Here $D_0$ stands for the set of  finite linear combinations of indicator functions of closed balls and $d$ denotes the diagonal set on $S\times S$.
		We assume that $J$ satisfies the following conditions:  
				\begin{enumerate}[({A}.1)]
							\item  \label{cond: A.1} For all $k \in \Z$ and $i \in S^k$, 
									$\int_{B_i^k\times (B_i^k)^c} J(x,y)\mu(dx) \mu(dy) < \infty.$ 		
							\item \label{cond: A.2}For all $(x,y) \in S\times S \setminus d$,
								$J(x,y)=J(y,x).$
					\end{enumerate}
			 In the above setting, $(\mathcal E, \mathcal F)$ is a regular Dirichlet form (See Section \ref{sec: proof of main result}).
			There exists a Hunt process $(\mathcal M_t, X_t, \mathbb P_x)$ on $S$ associated with $(\mathcal E, \mathcal F)$ (See, e.g.,\cite[Theorem 7.21]{Fukushima:2010uh}). 
			We write $(\mathcal E^k, \mathcal F^k)$ for  the following bilinear form:
							\begin{align*}
								\mathcal E^k (u,v)= \frac{1}{2}\sum_{i,j \in S^k} ( u(i) - u(j) )(v(i)-v(j)) J^k(i,j) \mu^k(i)\mu^k(j),								
							\end{align*}
						with its domain $\mathcal F^k =\overline{C_0^k}^{\mathcal E_1^k},$
	where $C^k_0$ denotes the set of functions on $S^k$ with finite support, $\mu^k(i):=\mu(B_i^k)$ for $i \in S^k$, and
		\begin{align*}
							J^k (i,j):= 
								\begin{cases}\dis
									\frac{1}{\mu^k(i)\mu^k(j)} \int_{B_i^k \times B_j^k } J(x, y) \mu(dx)\mu(dy) \quad (i \neq j),
									\\
									0 \quad (i=j).
								\end{cases}
						\end{align*}
	We call $(\mathcal E^k, \mathcal F^k)$ the {\it averaged Dirichlet form $($of level $k$$)$}. The averaged Dirichlet form $(\mathcal E^k, \mathcal F^k)$ is also a regular Dirichlet form, and let $(X^k_t, \mathbb P^k_i)$ be a Hunt process on $S^k$ associated with $(\mathcal E^k, \mathcal F^k)$.
	Now we obtain the following result:
		\begin{thm}  \label{thm: 230817Moscoconvergence}
					Suppose (A.\ref{cond: A.1}) and (A.\ref{cond: A.2}).
					Then the averaged Dirichlet form $(\mathcal E^k, \mathcal F^k)$
					is Mosco convergent in the generalized sense to $(\mathcal E, \mathcal F)$ as $k \to \infty$.
				\end{thm}
		The definition of the Mosco convergence in the generalized sense will be given in Section \ref{sec: Mosco convergence of averaged} following Chen--Kim--Kumagai \cite[Definition 8.1]{Chen:2013gu}.
		Theorem \ref{thm: 230817Moscoconvergence} is quite a general result applicable to very wide classes of symmetric Markov processes on $S$.
			For example, the assumptions of (A.\ref{cond: A.1}) and (A.\ref{cond: A.2}) are satisfied by the class constructed 
			by Albeverio--Karwowski \cite{Albeverio:2008uk},  the more general class constructed by Kigami \cite{KIGAMI:wb}
			and the new class constructed by this paper (called {\it mixed class}). See Section \ref{subsec: Mixed class} for details.
			 
	We want to know when $X^k$ converges weakly to $X$.
	We consider the following condition for tightness of $\{X^k\}_{k\in \Z}$:
	\begin{enumerate}[({A}.1)]
	\setcounter{enumi}{2}
							\item  \label{cond: A.3} For any  $k_1 \in \Z$, 
								\begin{align*}
									\sup_{k \ge k_1}\sup_{x\in S} \frac{1}{\mu(B_x^k)}\int_{B_x^k\times (B_x^{k_1})^c}J(y,z)\mu(dy)\mu(dz)<\infty.
								\end{align*}					
							\end{enumerate}
	To ensure conservativeness of $(\mathcal E, \mathcal F)$, we introduce the following condition:
	\begin{enumerate}[({A}.1)]
			\setcounter{enumi}{3}
			\item \label{cond: A.4}  $\sup_{x \in S}\int_{y:\rho(x,y)\ge1}J(x,y)\mu(dy) < \infty$. 			
	\end{enumerate} 
	Note that, since the ultrametric inequality does not allow processes to exit from a unit ball only by jumps whose sizes are smaller than one, we do not need other assumptions for small jumps and volume growth
	conditions for conservativeness. See Theorem \ref{thm: conservative}.
	
	Let $0<T<\infty$ and let $\mathbb D_S[0,T]$ denote the set of right-continuous paths $[0,T] \to S$ having left limits, which equipped with the Skorokhod topology.
	Let $C^+_0(S)$ denote the family of non-negative real-valued continuous functions with compact support on $S$.
	For each $\psi \in C_0^+(S)$, define the function $\psi^k$ on $S^k$ as 
			$\psi^k(i) := \frac{1}{\mu^k(i)}\int_{B_i^k}\psi(x)\mu(dx).$
	Define
		\begin{align*}
				\mathbb P_{\psi}(\cdot)=\frac{1}{\mu(\psi)}\int_{S}\mathbb P_{x}(\cdot)\psi(x)\mu(dx) \quad \text{and} \quad \mathbb P^k_{\psi^k}(\cdot)=\frac{1}{\mu^k(\psi^k)}\sum_{i\in S^k}\mathbb P^k_{i}(\cdot)\psi^k(i)\mu^k(i),
		\end{align*}
		where 
				$\mu(\psi)=\int_S\psi(x)\mu(dx) \quad \text{and}\quad \mu^k(\psi^k)=\sum_{i\in S^k}\psi^k(i)\mu^k(i).$
		Let $(X, \mathbb P_{\psi})$ and $(X^k, \mathbb P_{\psi^k}^k)$ are called {\it Hunt processes with the initial distributions $\psi\mu$ and $\psi^k\mu^k$}, respectively.
	Now we obtain the following main theorem:
	 \begin{thm} \label{weakconvergence}
	 	Suppose (A.\ref{cond: A.1})--(A.\ref{cond: A.4}). Let $\psi \in C_0^+(S)$ and $0 < T < \infty$.
		Then, as $k \to \infty$, the Markov chain $(X^k, \mathbb P_{\psi^k}^k)$
		converges in law on $\mathbb D_{S}[0,T]$ to the Hunt process $(X, \mathbb P_{\psi})$.
	\end{thm}
		The initial distributions are restricted to be absolutely continuous with respect to the reference measure $\mu$ because we use the Lyons--Zheng decomposition in the proof of tightness.
	 	We follow Chen--Kim--Kumagai \cite{Chen:2013gu} for an idea to prove tightness.
			\begin{rem}\normalfont
	There are some related works:
		\begin{enumerate}
		\item[(i)] The result and its proof of Theorem \ref{weakconvergence} are very similar to \cite[Theorem 6.1]{Chen:2013gu}. The formulations, however,  are so different that we cannot reduce Theorem \ref{weakconvergence} to \cite[Theorem 6.1]{Chen:2013gu}. In \cite{Chen:2013gu}, the authors studied a metric measure space $(E,\rho,m)$ under several conditions
		where {\it approximation graphs} $\{(V_k,\Theta_k)\}_{k\in \Z}$ can be constructed. They proved that  Markov chains $X^k$ on $(V_k, \Theta_k)$ converge weakly to a Markov process $X$ on $(E, \rho, m)$, 
		provided that $X^k$'s satisfy several conditions. Especially they assumed that the  jump density of $X^k$ is controlled by the graph metric of $(V_k,\Theta_k)$.
		In our setting, however, the approximation set $S^k$ does not need any graph structures and actually we essentially use only topological properties in the proof of Theorem \ref{weakconvergence}. 
		
		\item[(ii)] In Mart\'inez--Remenik--Mart\'in \cite{MARTINEZ:2007wl}, they investigated a convergence of  Markov processes on a compact ultrametric space.
		Let $T$ be a locally-finite rooted tree and cut this tree on each finite level $k$ (we write $T^k$ for the cut tree).
		They gave some random walks $W$ on $T$ and $W^k$ on $T^k$. 
		They showed that $W$ and $W^k$ induced Markov processes $X$ and $X^k$ on their Martin boundaries $S$ and $S^k$, which $S$ and $S^k$ are compact ultrametric spaces.
		Then, they showed that $X^k$ converges weakly to $X$ as $k\to \infty$.  
		In our setting, the state space $S$ is not necessarily compact, and thus tightness of $\{X^k\}_k$ is not obvious.
		Moreover, our Markov processes need not to be induced by random walks on trees.
		\item[(iii)] In Yasuda \cite{Yasuda:2006wr},  a different type of convergence theorem of L\'evy processes on local fields $K$ was proved.
		 Let $X_t$ be a semi-stable process on $K$ with an epoch $a<1$ and a span $b$. The author showed that, if $\xi_i$ is identically distributed as $X_1$, $a_n=a^{-n}$, and $b_n=b^n$,
		 then $1/b_n\sum_{i=1}^{[a_nt]}\xi_i$ as $n \to \infty$ converges weakly to $X$, where $[a_nt]$ stands for the integral part of $a_nt$. This result uses the algebraic structures of local fields. We do not know the relationship of our result and their result. 
	\end{enumerate}
	\end{rem}
	
	We introduce an interesting property of Markov processes on $S$.
	Let $\pi^k$ be the canonical map from $S$ to $S^k$.
	We call the following property {\it projection Markov property of level $k$} (Write (pMp)$_k$ for short):
	\begin{description}
		\item[(pMp)$_k$] There exists a transition probability $\tilde{p}^k_t(x,y)$ such that $\mathbb P_x(\pi^k\circ X_{t+s}=y|\mathcal M_s)=\tilde{p}_t^k(\pi^k\circ X_s,y)$ for quasi-every $x\in S$ and all $y \in S^k$;
		consequently, the projected process $\pi^k \circ X_t$ under $(\mathcal M_t, \mathbb P_x)$ is also a Markov process for quasi-every $x \in S$. 
	\end{description}
	This  property cannot be expected in the usual Euclidean spaces. 
	For example, let $X$ be the $2$-dimensional Brownian motion and let us take a countable partition of $\R^2$ by $\{[n,n+1)\times[m,m+1)\}_{n,m\in\Z}$. Let $\R^2/\sim$ denote the quotient space 
	with respect to this partition and let $\pi^1$ be the canonical map from $\R^2$ to $\R^2/\sim$. 
	Then, we can see that $\pi^1 \circ X$ is not Markov.
	
	We consider the following question:
		\begin{description}
			\item[(Q2)] 
			When do Markov processes  have the projection Markov property of level $k$?
		\end{description}
	We introduce several additional conditions:
	 For a fixed $k$, 
	 \begin{enumerate}
				\item[(BC)$_k$] \label{cond: A.3k} (Ball-wise constance of level $k$)
					For each $i,j \in S^k$ with $i \neq j$, 
						\begin{align}
							J|_{B_i^k \times B_j^k} \equiv C_{ij}^k \notag
						\end{align}
					for some constant $C_{ij}^k$ depending only on $i,j,k$.
			\item[(BC)$_\infty$] 
					The condition (BC)$_k$ is satisfied for all $k \in \Z$.
		\end{enumerate}
	Note that these conditions are natural for ultrametric spaces because there are many locally constant functions on ultrametric spaces.	
	
	Let $\{P_t^k\}_{t \ge 0}$ be the transition semigroup on $L^2(S^k;\mu^k)$ associated with the averaged Dirichlet form $(\mathcal E^k, \mathcal F^k)$ and $E^k$ be the extension operator 
				defined in \eqref{E^kdef}.

		\begin{thm} \label{thm: P=A}
			\begin{enumerate} We have the following:
			\item[(1)] Suppose (A.1) and (A.2). The following two assertions are equivalent:
				\begin{enumerate}
				\item[(i)] The Hunt process $(X_t, \mathbb P_x)$ has the projection Markov property of level $k$.
				\item[(ii)] $P_tE^kf=E^kP_t^kf$ for all $f \in L^2(S^k;\mu^k)$ and $t \ge 0$.
				\end{enumerate}
				If either one (and hence each) of the assertions {\bf(i)} and {\bf (ii)} above holds, the projected process $(\pi^k\circ X, \mathbb P_x)$ is equal in law to 
				$(X^k, \mathbb P_i^k)$.
			\item[(2)] Suppose (A.1), (A.2) and (BC)$_k$. Then either one (and hence each) of the assertions {\bf(i)} and {\bf (ii)} above holds. 
								\end{enumerate}
				\end{thm}
		
			In Dynkin \cite{Dynkin:1965tg}, Rogers--Pitman \cite{Rogers:1981tob} and Glover \cite{Glover:1991to}, the functions preserving Markov property 
			were called {\it Markov functions} and they studied several sufficient conditions of different types for functions to be Markov functions.
			Theorem \ref{thm: P=A} asserts that the canonical projection $\pi^k$ is a Markov function under (A.1), (A.2) and (BC)$_k$. 
			The key to the proof is to verify Dynkin's sufficient condition \cite{Dynkin:1965tg}.
			
			As a corollary of Theorem \ref{thm: P=A}, under the condition (BC)$_{\infty}$, we can show that $X^k$ converges to $X$ almost sure for $k\to \infty$, which is the stronger result than Theorem \ref{weakconvergence}:
		\begin{cor} \label{cor}
		Suppose (A.\ref{cond: A.1})--(A.\ref{cond: A.2}) and (BC)$_{\infty}$. Then $(X^k, \mathbb P_i^k)$ is equal in law to $(\pi^k\circ X, \mathbb P_x)$ for all $k$, and $\pi^k\circ X_t$ converges to $X_t$ as $k \to \infty$ uniformly on compact intervals in $t$ $\mathbb P_x$-almost everywhere for quasi-every $x\in S$.
		\end{cor}
		
			This paper is organized as follows. In Section \ref{sec: proof of main result}, we give preliminary facts for ultrametric spaces and Dirichlet forms on ultrametric spaces. In Section \ref{sec: Mosco convergence of averaged},  
			we prove Theorem \ref{thm: 230817Moscoconvergence}. First, we recall the definition of Mosco convergence in the generalized sense given in Chen--Kim--Kumagai \cite{Chen:2013gu}, 
			and we introduce extension and restriction operators in our setting. Second, we relate $(\mathcal E, \mathcal F)$ and the averaged Dirichlet form $(\mathcal E^k,\mathcal F^k)$ by using the 
			extension operator. 
			Finally, we complete the proof of Theorem \ref{thm: 230817Moscoconvergence}.
			In Section  \ref{sec: proof of 2}, we prove Theorem \ref{weakconvergence}. 
			First, we give a sufficient condition for conservativeness and obtain convergence of finite-dimensional distributions. Second,
	we prove tightness of $X^k$, and we complete the proof of Theorem \ref{weakconvergence}.
	In Section \ref{pf: thmP=A}, we prove Theorem \ref{thm: P=A}.
	In Section \ref{sec: examples}, we introduce a mixed class, a kind of generalization of the class constructed in \cite{KIGAMI:wb}.
				
	Throughout this paper, we denote by $\Z, \N$ and $\N_0$ the set of integers, positive integers and non-negative integers, respectively.
	Sometimes we write $\sum_{i}A_i$ for $\bigcup_iA_i$ whenever $\{A_i\}_{i \in \N}$ is disjoint.
	We write $C_0(S)$, $C_b(S)$ and $C_{\infty}(S)$ for the class of real-valued continuous functions on $S$ with compact support, bounded and vanishing at infinity, respectively.
	Write $C_0^k$ for the class of continuous functions on $S^k$ with finite support.
	Define, for $0<T<\infty$,
		\begin{align*}
			\mathbb D_S[0,T] = \{f :[0,T] \to S: f \text{ is a right-continuous function having left limits}.\},
		\end{align*}
		and equip $\mathbb D_S[0,T]$ with the Skorokhod topology (cf. e.g., \cite{Ethier:2009tg}).
		
	\section{Preliminary facts}\label{sec: proof of main result} 
		We recall the following facts of ultrametric spaces:
		\begin{fact} \label{fact}
			\begin{enumerate}
				\item [(i)]  If $B_a^k \cap B_b^{k} \neq \emptyset$, then $B_a^k = B_b^k$. \label{Ball}
				\item[(ii)] The metric $\rho$ is constant on  $B_a^k \times B_b^k$, whenever $B_a^k \neq B_b^k$.
				\item[(iii)] Any closed ball is open.
				\item[(iv)] Indicator functions of closed (open) balls are continuous.
			\end{enumerate}
		\end{fact}
		 We shall utilize the following basic properties of ultrametric spaces under the conditions (U.\ref{cond: U.0})--(U.\ref{cond: U.2}).
			\begin{prop} \label{prop: basic of ultra-metric}
			Suppose (U.\ref{cond: U.0})--(U.\ref{cond: U.2}). Then the following assertions hold: 
				\begin{enumerate}
					\item[(i)] For any $k \in \Z$ and $a \in S$, there exists a finite subset $\{a_i\}_i \subset B_{a}^k$ such that $\{B_{a_i}^k\}_i$ is disjoint and 
							$B_a^k = \sum_{i} B_{a_i}^{k+1}.$
						 If, moreover, there exists another such  finite subset $\{b_i\}_i$, it follows that $B_{a_i}^k = B_{b_i}^{k+1}$ for all $i$ after suitable rearrangement.
					\item[(ii)] For any $k \in \Z$, there exists a countable subset  $\{a_i\}_i \subset E$ such that $\{B_{a_i}^k\}_{i}$ is disjoint and 
							$S = \sum_{i} B_{a_i}^k.$ 
						  If, moreover, there exists another such a countable subset $\{b_i\}_{i}$, it follows that $B_{a_i}^k = B_{b_i}^k$ for all $i \in \N$ after suitable rearrangement.
					\item[(iii)] $(S,\rho)$ is separable.
				\end{enumerate}
			\end{prop}
			\begin{proof}
					{\bf(i)} Let $a_0 = a \in S$. We can take the following open covering of $B_a^k$:
							$B_a^k = \bigcup_{x \in B_a^k}B_x^{k+1}.$
					By {(U.\ref{cond: U.1}}) and (i) of Fact \ref{fact}, we can extract a finite subcover $B_a^k = \sum_i B_{a_i}^{k+1}$.
					The latter assertion is obvious by (i) of Fact \ref{fact}.
					
					{\bf(ii)} Let $x \in S$ be fixed. By (i), there exists a finite subset $\{a_i^1\}_i \subset B_i^{k-1}$ such that  
							$B_x^{k-1} = \sum_{i} B^k_{a_i^1}.$
						By using (i) again, there exists a finite subset $\{a_i^2\}_i \subset B_x^{k-2}$ such that
							$B_x^{k-2} \setminus B_x^{k-1}= \sum_{i} B^k_{a_i^2}.$
						Using this argument inductively, we see that							
								$S = \bigcup_{l=1}^{\infty}B_x^{k-l}=\sum_{i \in \N} B_{a_i}^k.$
						The latter part of the statement can be shown by the same argument as the proof of (i). 
						
					{\bf(iii)} Let $\{a_i^k\}_{i \in \N}$ be a countable subset of $S$ such that $S = \sum_{i \in \N} B_{a_i^k}^k$. 
					Define $S^k=\{a_i^k\}_{i}$ and $S= \bigcup_{k \in \Z}S^k$. Clearly, $S$ is countable.
					For any $x \in S$, there exists a unique sequence $\{i(k)\}_{k \in \Z}$ such that
							$ ...\supset B_{a_{i(k)}}^k \supset B_{a_{i(k+1)}}^{k+1} \supset ... \supset \{x\}.$
					This means that $a_{i(k)} \to x$ as $k \to \infty$. We have completed the proof.
				\end{proof}
		We prepare several classes of functions on $S$ and $S^k$. Let $\1_i^k$ denote the indicator function of $B_i^k$. 
		Let us define
			\begin{align}
				D^k=\biggl\{\sum_{i \in S^k} c_i\1_i^k: c_i \neq 0 \text{ only for finite $i$'s.} \biggl\}, \notag
			\end{align}
		for each $k \in \Z$.
		It is obvious by {(i)} of Proposition \ref{prop: basic of ultra-metric} that $D^k \subset D^{k+1}$.
		We denote 
			\begin{align}
				D_0= \bigcup_{k \in \Z} D^k \label{D0}.
			\end{align}
		In other words, $D_0$ is the set of finite linear combinations of indicator functions of balls.  
		Since $D^k \subset D^{k+1}$,   for each $u \in D_0$, we write $m(u)$ for
			\begin{align}
				m(u) = \inf\{k: u \in D^{k}\}. \label{def: m(u)}
			\end{align}
		The class $D_0$ is a dense subset both of $C_0(S)$ with respect to the uniform norm and of $L^2(S;\mu)$ with respect to the $L^2$-norm. 
		
						Assume (A.\ref{cond: A.1}) and (A.\ref{cond: A.2}).  Then $(\mathcal E, \mathcal F)$ and $(\mathcal E^k, \mathcal F^k)$ are  symmetric regular Dirichlet forms.
						We refer the readers to \cite[Example 1.2.4]{Fukushima:2010uh}, for example.

		
		\section{Proof of Theorem \ref{thm: 230817Moscoconvergence}} \label{sec: Mosco convergence of averaged}
			We adopt the generalized Mosco convergence 
			following Chen-Kim-Kumagai \cite[Appendix]{Chen:2013gu}.
			For $k \in\Z$, let $(\mathcal H^k, \langle \cdot, \cdot \rangle_k )$ and $(\mathcal H, \langle \cdot, \cdot \rangle)$ be Hilbert spaces whose norms 
			are denoted by $ \| \cdot \|_k $ and $\| \cdot \|$. Suppose that $(a^k, \mathcal D [a^k])$ and $(a, \mathcal D[a])$ are positive densely defined closed symmetric contraction bilinear
			forms on $\mathcal H^k$ and $\mathcal H$, respectively. We extend the definition of $a^k(u,u) $ to every $u \in \mathcal H^k$ by defining
			$a^k(u,u) = \infty$ for $u \in \mathcal H^k \setminus \mathcal D[a^k]$. Similar extension is done for $a$ as well.
			
			Let $E^k:\mathcal H^k \to \mathcal H$ and $\Pi^k : \mathcal H \to \mathcal H^k$ be bounded linear operators.
			If the following conditions are satisfied, we call $E^k$ the {\em extension operator} and $\Pi^k$ the {\em restriction operator}, respectively:
			\begin{enumerate}
				\item[(ER.1)]
						$\langle \Pi^k u, v \rangle_k = \langle u, E^k v \rangle$ for $u \in \mathcal H$ and $v \in \mathcal H^k.$
				
				\item[(ER.2)]
					$\Pi^k E^k u = u $ for $u \in \mathcal H^k$.	
				
				\item[(ER.3)]
						$\sup_{k \in \Z}\| \Pi^k \|_{op} < \infty$ ,
				where $\| \cdot \|_{op} $ denotes the operator norm.
				\item[(ER.4)]  For every $u \in \mathcal H$, 
						$\lim_{k \to \infty} \| \Pi^ku \|_k = \| u \|$.
			\end{enumerate}	
			It follows immediately that $E^k$ is isometry, i.e., $\langle E^k u, E^k v \rangle = \langle u, v \rangle_k $, for every $k$ and $u, v \in \mathcal H^k , k \in \Z$.
	
	Now the generalized Mosco convergence is defined as follows:
				\begin{defn} \label{defn: Mosco}Under the above setting, we say that the closed bilinear form $a^k$ is {\em Mosco-convergent to $a$ in the generalized sense} if 
				the following conditions are satisfied $:$
					\begin{enumerate}
						\item[(i)] If $u_k \in \mathcal H^k , u \in \mathcal H $ and $E^k u_k \to u$ weakly in $\mathcal H $, then
									\begin{align}
										\liminf_{k \to \infty} a^k (u_k, u_k) \ge a(u,u). 		\label{ineq: Mosco1}
									\end{align} 
						\item[(ii)]  For every $u \in \mathcal H $, there exists a sequence $\{u_k\}$ such that $u_k \in \mathcal H^k$ for all $k$, $E^k u_k \to u $ strongly in $\mathcal H $
								and 
									\begin{align}
										\limsup_{k \to \infty} a^k (u_k, u_k) \le a(u,u).
									\end{align}
					\end{enumerate}
				\end{defn}
		Now we introduce extension and restriction operators in our setting.	
				Set $\mathcal H^k = L^2(S^k; \mu^k)$ and $\mathcal H = L^2(S;\mu)$ equipped with $L^2$-inner products:
				\begin{align*}
					\langle u,v\rangle_k:=\sum_{i\in S^k}u(i)v(i)\mu^k(i) \quad\text{and}\quad \langle u,v\rangle :=\int_Su(x)v(x)\mu(dx).
				\end{align*}
				Let us define the bounded linear operators $E^k: L^2(S^k;\mu^k) \to L^2(S;\mu)$ and $\Pi^k: L^2(S;\mu) \to L^2(S^k;\mu^k)$ as follows:
				\begin{align}
					&\Pi^ku ( i ) =  \frac{1}{\mu^k( i )}\int_{B_i^k}u(x)\mu(dx), \label{Pi^kdef}
					\\
					&E^kv(x) = v([x]_k)=\sum_{i \in S^k}v(i)\1_{B_i^k}(x),   \label{E^kdef}
				\end{align}
			for $u\in L^2(S;\mu)$ and $v \in L^2(S^k; \mu^k)$.
			It is easy to check that $E^k$ and $\Pi^k$ defined in \eqref{Pi^kdef} and \eqref{E^kdef} satisfy (ER.1)--(ER.4) (See \cite[Lemma 4.1]{Chen:2013gu}).
				
		The following proposition relates the averaged Dirichlet form $(\mathcal E^k, \mathcal F^k)$ with $(\mathcal E, \mathcal F)$.
				\begin{prop}\label{thm: EkEEk}
				Suppose (A.\ref{cond: A.1}) and (A.\ref{cond: A.2}) hold. Then, the following assertions hold:
					\begin{enumerate}
						\item[(i)]For all $u \in C_0^k$, $E^ku \in D_0$,
						\item[(ii)]For all $u \in \mathcal F^k$, $E^ku \in \mathcal F$,
						\item[(iii)]For all $u \in \mathcal F^k$, $\mathcal E^k(u,u) = \mathcal E(E^k u,E^k u)$.
					\end{enumerate}
				\end{prop} 
					\begin{proof} (i) is clear by definition.
					
					(ii) and (iii): Let $u \in C_0^k$. Let $d^k$ denote the diagonal set of $S^k \times S^k$. Then, we have
						\begin{align}
							\mathcal E^k ( u, u)
							&= \frac{1}{2}\sum_{(i,j) \in S^k \times S^k \setminus d^k}\left|  u (i)  - u (j)  \right|^2  J^k(i,j) \mu ^k(i) \mu^k(j) \notag
							\\ 
							& =  \frac{1}{2}\sum_{(i,j) \in S^k \times S^k\setminus d^k}\left|  u (i)  - u (j)  \right|^2 \int_{B_i^k\times B_j^k} J(x,y) \mu (dx)\mu(dy)   \notag
							\\
							&  =  \frac{1}{2}\sum_{(i,j) \in S^k \times S^k \setminus d^k} \int_{B_i^k \times B_j^k}\left| E^ku (x) - E^ku (y)\right|^2 J(x,y) \mu (dx) \mu(dy) \notag
							\\
							& =   \frac{1}{2}\int_{S \times S \setminus d}\left| E^ku (x) - E^ku (y)\right|^2 J(x,y) \mu (dx) \mu(dy) \notag
							\\
							& =\mathcal E ( E^k u , E^k u ). \notag
						\end{align}
						Thus we see that 
							\begin{align}
								\mathcal E^k(u,u) = \mathcal E(E^k u,E^k u) \quad \text{for} \quad u \in C_0^k. \label{eq: EkEEk}
							\end{align}
						Let $u \in \mathcal F^k$. Then, we can take a  $\mathcal E_1^k$-Cauchy sequence $\{u_n\}_{n \in \N} \subset C_0^k$ 
						such that $u_n\to u$ in $L^2(S^k;\mu^k)$ and 
							\begin{align}
								\lim_{n \to \infty}\mathcal E^k(u_n,u_n) = \mathcal E^k(u,u). \label{eq: E^kun=E^ku}
							\end{align} 
						By equality (\ref{eq: EkEEk}), we see that  $\{u_n\}_{n \in \N}$ is  a Cauchy sequence with respect to $\mathcal E(E^k\cdot, E^k\cdot)$. By this fact and 
						$E^ku_n \to E^ku$ in $L^2(S;\mu)$, we have
						\begin{align}
							\lim_{n \to \infty}\mathcal E(E^ku_n,E^ku_n) = \mathcal E(E^ku,E^ku). \label{eq: EE^kun=EE^ku}
						\end{align} 
						Hence, we obtain $E^ku \in \mathcal F$. Furthermore, by equalities (\ref{eq: E^kun=E^ku}), \eqref{eq: EkEEk} and (\ref{eq: EE^kun=EE^ku}), we have
						$\mathcal E^k(u,u) = \lim_{n \to \infty}\mathcal E^k(u_n,u_n) = \lim_{n \to \infty}\mathcal E(E^ku_n,E^ku_n) = \mathcal E(E^ku, E^ku).$
					We have completed the proof.
					\end{proof}

				For the proof of Mosco convergence of $(\mathcal E^k, \mathcal F^k)$ to $(\mathcal E, \mathcal F)$, we need Fatou's lemma for $(\mathcal E, \mathcal F)$.
				\begin{lem} \label{lem: Fatou}
					If $u_n \in \mathcal F$ converges $\mu$-almost everywhere to $u \in L^2(S;\mu)$,
					then $\mathcal E(u,u) \le \liminf _{n}\mathcal E(u_n,u_n).$
				\end{lem}
						This is a direct application of  Proposition 1 of Schmuland \cite{Schmuland:1999vr}. So we omit the proof.
				
		Suppose that (i) of Definition \ref{defn: Mosco} holds. Then the following Lemma is a sufficient condition for (ii) of Definition \ref{defn: Mosco} (See \cite[Lemma 8.2]{Chen:2013gu}): 
				\begin{lem} \label {lem: alternative of Mosco} Suppose that (A.\ref{cond: A.1}) and (A.\ref{cond: A.2}) hold.
				Then the following conditions hold$:$
				\begin{enumerate}
					\item[(i)] $D_0$ is dense in $\mathcal F$ with respect to $\mathcal E_1$-norm.
					\item[(ii)] For every $ u \in D_0$, it holds that $\Pi^k u \in C_0^k $.
					\item[(iii)] For every $u \in D_0 $, 
							$\limsup_{k \to \infty}\mathcal E^k ( \Pi^k u, \Pi^k u ) = \mathcal E ( u, u ).$
				\end{enumerate}
			\end{lem}			
					\begin{proof}
						(i) is clear by the definition of $(\mathcal E , \mathcal F )$ and 
						(ii) is clear by the definition of $(\mathcal E^k, \mathcal F^k)$.
						(iii): For  $k \ge m(u)$,  we have
								$E^k \Pi ^k u = u.$
						Hence, from Proposition \ref{thm: EkEEk}, we have
								$\limsup_{k \to \infty}\mathcal E^k ( \Pi^k u,  \Pi^k u ) = \limsup_{k \to \infty}\mathcal E( E^k \Pi^k u,  E^k \Pi^k u )  = \mathcal E ( u, u ).$
						The proof is complete.
					\end{proof}
				Now we show that $(\mathcal E^k, \mathcal F^k)$ is  Mosco convergent to $(\mathcal E, \mathcal F)$.
						\begin{proof}[Proof of Theorem \ref{thm: 230817Moscoconvergence}]
							By  Lemma \ref{lem: alternative of Mosco},  we only have to show condition (i) of Definition \ref{defn: Mosco}. It suffices to prove inequality
							(\ref{ineq: Mosco1}) only for all sequences $u_k \in L^2( S^k ; \mu^k ) $
							 such that $E^k u_k$ converges weakly to
							$u \in L^2( S ; \mu )$ and $\liminf _{k \to \infty} \mathcal E^k ( u_k, u_k ) =:L < \infty $. 
							Taking a subsequence if necessary, we may 
							 assume that $\lim _{k \to \infty} \mathcal E^k (u_k, u_k ) = L$.
							  Since $E^ku_k$ converges to $u$ weakly, $\{E^ku_k\}_{k \in \Z}$ is a bounded sequence in $L^2(S;\mu)$.
							 
							By the Banach-Saks Theorem (See, e.g. \cite[Theorem A.4.1]{Chen:2011tb}), by taking a subsequence if necessary, 
							we may assume that $v_k:= \frac{1}{k} \sum_{i=1}^k  E^i u_i $ converges to some 
							$v_{\infty} \in L^2 (S;\mu )$. Since $E^k u_k $ converges weakly to $u$,  
							we see that $v_{\infty}= u $ $\mu$-a.e. on $S$.  
								Then, by the triangle inequality with respect to $\mathcal E(\cdot,\cdot)^{\frac{1}{2}}$ and by Proposition \ref{thm: EkEEk}, we have 
									\begin{align*}
										\mathcal E ( v_k, v_k ) ^{ \frac{1}{2} } 
									\le  \frac {1}{k} \sum _{i = 1}^k \mathcal E ( E^i u_i, E^i u_i )^{ \frac{1}{2}} 
									=\frac {1}{k} \sum _{i = 1}^{k} \mathcal E^i ( u_i, u_i ) ^{ \frac{1}{2}}
									 \to L^{\frac{1}{2}} \quad (k\to \infty).
									 \end{align*}
							In addition, from Lemma \ref{lem: Fatou}, we have
									$\mathcal E (u,u)
									 \le  \liminf_{k \to \infty} \mathcal E (v_k, v_k )  \le L.$
							Hence we have completed the proof.
						\end{proof}	
	\section {Proof of Theorem \ref{weakconvergence}}	\label{sec: proof of 2}
	Now we give a sufficient condition for conservativeness:
			\begin{thm} \label{thm: conservative}
				Suppose that  (A.\ref{cond: A.1}), (A.\ref{cond: A.2}) and (A.\ref{cond: A.4}) hold.  Then $(\mathcal E,\mathcal F)$ and $(\mathcal E^k, \mathcal F^k)$ are conservative.
			\end{thm}
	\begin{proof}
		Let $J'(x,y)=J(x,y)\1_{\rho(x,y)\le1}$ and $J''(x,y)=J(x,y)\1_{\rho(x,y)\ge1}$. Define $(\mathcal E', \mathcal F')$ and $(\mathcal E'', \mathcal F'')$ be the corresponding Dirichlet forms to 
		$J'$ and $J''$, respectively.
		Using (A.\ref{cond: A.1}), (A.\ref{cond: A.2}), (A.\ref{cond: A.4}) and the fact that $D_0$ is a core for $(\mathcal E, \mathcal F)$,  and developing the same argument as \cite[Theorem 2.2]{GHM}, 
		we see that the conservativeness of $(\mathcal E, \mathcal F)$ is equivalent to that of the truncated Dirichlet form $(\mathcal E', \mathcal F')$.
	
		Now it suffices to show that $(\mathcal E', \mathcal F')$ is conservative. However, by the ultrametric inequality, the corresponding process cannot exit from a unit ball 
		only by small jumps (i.e., jumps whose sizes are smaller than $1$). Since the jump density $J'$ has only small jumps, $(\mathcal E', \mathcal F')$ is conservative.
		Using Proposition \ref{thm: EkEEk}, the conservativeness of $(\mathcal E^k, \mathcal F^k)$ can be shown by the same argument. We finish the proof.
			\end{proof}	
							
			Now we obtain convergence of  finite-dimensional distributions by Theorem \ref{thm: 230817Moscoconvergence} and Theorem \ref{thm: conservative}.
				\begin{cor} \label{theorem: fd}
				Suppose that  (A.\ref{cond: A.1})--(A.\ref{cond: A.2}) and (A.\ref{cond: A.4}) hold.
				Then, for any $\psi \in C_0^{+}(S)$, finite dimensional distributions of the Hunt process $(X_t^k, \mathbb P_{\psi}^k)$ converge to those of 
				$(X_t,\mathbb P_{\psi})$ as $k \to \infty$.
			\end{cor}
			\begin{proof}
				By Theorem \ref{thm: conservative}, $(\mathcal E, \mathcal F)$ is conservative.
				Thus, by Theorem \ref{thm: 230817Moscoconvergence}, we can do the same argument as Chen--Kim--Kumagai \cite[Theorem 5.1]{Chen:2013gu} and 
				 we omit the proof.
			\end{proof}
	
			For proving  tightness, we adopt the same strategy as \cite{Chen:2013gu}. The following lemma plays a key role in proving tightness, which corresponds to \cite[Lemma 3.3]{Chen:2013gu} and the following proof is a modification of the proof of \cite[Lemma 3.3]{Chen:2013gu}.
		\begin{lem} \label{230817tightlemma}
		Suppose that (A.\ref{cond: A.1})--(A.\ref{cond: A.3}) hold.
			Then, for any $g \in D_0$, there exist $C >0$ and $k_0 \in \Z$ such that,  for any $k \ge k_0$ and any $0 \le t \le s < \infty$,
			\begin{align*}
				\int_s^t \sum_{j \in S^k}(g(X_u^k)-g(j))^2J^k(X_u^k,j)\mu^k(j)du \le C(t-s).
			\end{align*}
		\end{lem}
			\begin{proof} Let $C_g := \max_{x,y \in supp(g)}|g(x)-g(y)|$.
				 Take $k_0 = \max\{m(g), k_1\}$, where $k_1$ has been defined in the condition (A.\ref{cond: A.3}) and $m(g)$ has been defined in \eqref{def: m(u)}.
				  Note that $|g(x)-g(y)|=0$ if $r(x,y) \ge k_0$, i.e., $\rho(x,y) \le \phi(k_0)$. For any $k \ge k_0$, we have that
					\begin{align}
					&\sup_{i \in S^k} \sum_{j \in S^k}(g(i)-g(j))^2J^k(i,j)\mu^k(j) \notag
					\\
					=&\ \sup_{i \in S^k} \Bigl(\sum_{j : g(j)=0}g(i)^2J^k(i,j)\mu^k(j)+\sum_{j:g(j)\neq 0}(g(i)-g(j))^2J^k(i,j)\mu^k(j)\Bigr) \notag
					\\
					\le&\ \|g\|_{\infty}^2\sup_{i:g(i)\neq 0}\sum_{j:g(j)=0}J^k(i,j)\mu^k(j) +C_g^2\sup_{i:g(i)\neq 0 }\sum_{\substack{j:g(j)\neq 0\\ r(i,j)\le k_0}}J^k(i,j)\mu^k(j) \notag
					\\
					&+\ \|g\|_{\infty}^2\sup_{i :g(i)=0 }\sum_{j :g(j)\neq 0}J^k(i,j)\mu^k(j) \notag
					\\
					\le&\ (2\|g\|_{\infty}^2\vee C_g^2) \sup_{i:g(i)\neq 0}\sum_{\substack{j \in S^k\\ r(i,j)\le k_0}}J^k(i,j)\mu^k(j) \notag
					\\
					&+\ \|g\|_{\infty}^2\sup_{i:g(i)=0}\sum_{j:g(j)\neq 0}J^k(i,j)\mu^k(j) \notag
					\\
					=:&{\sf (I)_k}+{\sf (II)_k}. \notag
					\end{align}
			By (A.\ref{cond: A.3}), {$\sf (I)_k$} is bounded in $k$ because we have 
				\begin{align*}
					\sup_{i:g(i)\neq 0}\sum_{\substack{j \in S^k\\ r(i,j)\le k_0}}J^k(i,j)\mu^k(j)
					&=\sup_{i :g(i)\neq 0}\sum_{\substack{j \in S^k\\ r(i,j)\le k_0}}\frac{1}{\mu^k(i)}\int_{B_i^k\times B_j^k}J(x,y)\mu(dx)\mu(dy)
					\\
					&<\sup_{k \ge k_0}\sup_{i \in S^k}\frac{1}{\mu^k(i)}\int_{B_i^k\times (B_i^{k_0})^c}J(x,y)\mu(dx)\mu(dy) < \infty.
				\end{align*}
			By (A.\ref{cond: A.3}), {$\sf (II)_k$} is bounded in $k$ because we have
				\begin{align*}
					&\sup_{i:g(i)=0 }\sum_{j:g(j)\neq 0}J^k(i,j)\mu^k(j)
					\\
					&<\sup_{k \ge k_0}\sup_{i \in S^k}\sum_{j :g(j)\neq 0}\frac{1}{\mu^k(i)}\int_{B_i^k\times (B_i^{k_0})^c}J(x,y)\mu(dx)\mu(dy)
					< \infty.
				\end{align*}
				Thus we have completed the proof.
					\end{proof}

	Lemma \ref{230817tightlemma} enables us to follow the same argument as \cite[Proposition 3.4]{Chen:2013gu} to obtain the following Proposition:
		\begin{prop} \label{prop: 230817}
			Suppose that (A.\ref{cond: A.1})--(A.\ref{cond: A.4}) hold.
			Let $\psi \in C_0^{+}(S)$ and $0<T<\infty$. Then, for any finite subset $\{g_1, \cdot\cdot\cdot,g_m\} \subset D_0^{+}$, 
			the family of the laws of the processes $\{(g_1,\cdot\cdot\cdot,g_m)\circ X^k, \mathbb P^k_{\psi^k}\}_{k \in \Z}$
			is tight in $\mathbb D_{\R^m}[0,T]$.
		\end{prop} 
	Now we prove Theorem \ref{weakconvergence}.
	\begin{proof}[Proof of Theorem $\ref{weakconvergence}$]
	By Corollary \ref{theorem: fd} and Proposition	\ref{prop: 230817}, for any finite subset $\{g_1,g_2,...,g_m\} \subset D_0^+$, we have that 
			$\bigl((g_1,...,g_m)\circ X^k,\mathbb P_{\psi^k}^k\bigr)$ converges as $k\to \infty$ to $\bigl((g_1,...,g_m)\circ X, \mathbb P_{\psi}\bigr)$ in law on $\mathbb D_{\R^m}[0,1].$
		Since $D_0^{+}$ strongly separates points in $S$, by using \cite[Corollary 3.9.2]{Ethier:2009tg}, we complete the proof.
	 \end{proof}

	
	\section{Proof of Theorem \ref{thm: P=A}}  \label{pf: thmP=A}
		Let $\{P_t\}_{t\ge 0}$ (resp. $\{P_t^k\}_{t\ge 0} )$ and $\{G_{\alpha}\}_{\alpha\ge 0}$ (resp. $\{G_{\alpha }^k\}_{\alpha \ge 0}$)  denote the semigroups and resolvents corresponding
			 to the Dirichlet form $(\mathcal E, \mathcal F)$ (resp. the averaged Dirichlet form $(\mathcal E^k, \mathcal F^k)$), respectively.
			Define
					$ \mathcal E_{\lambda}(u,v)=\mathcal E(u,v) +\lambda\langle u,v\rangle$  for $\lambda \ge 0, u,v\in \mathcal F$, 
				and
					$\mathcal E^k_{\lambda}(u,v)=\mathcal E^k(u,v) +\lambda\langle u,v\rangle_k$ for $\lambda \ge 0, u,v\in \mathcal F^k$.
			Here recall that, in Section \ref{sec: Mosco convergence of averaged}, we have defined 	
					$\langle u,v\rangle =\int_Su(x)v(x)\mu(dx)$ and $\langle u,v\rangle_{k} =\sum_{i\in S^k}u(i)v(i)\mu^k(i)$.
		
					We now prove Theorem \ref{thm: P=A}.
					\begin{proof}[Proof of Theorem \ref{thm: P=A}]
					{\bf (1)} {\bf (i) $\Rightarrow$ (ii):} Let $\tilde{P}_t$ denote the transition semigroup of $(\pi^k\circ X, \mathbb P_x)$. 
					By (pMp)$_k$, we have 
						\begin{align}
							P_tE^kf(x) = \mathbb E_xf(\pi^k\circ X_t)
							&=\tilde{P}_t^kf(\pi^kx)=E^k\tilde{P}_t^kf(x). \label{pMp_k_1}
						\end{align}
						 Let $\tilde{\mathcal E}^k(u)=\lim_{t \downarrow 0}(1/t)(u-\tilde{P}_t^ku,u)$ for $u \in D^k$.
						Then, by the isometry of $E^k$, \eqref{pMp_k_1} and Proposition \ref{thm: EkEEk}, we have 
							\begin{align*}
								\tilde{\mathcal E}^k(u)=\lim_{t\downarrow0}\frac{(E^ku-E^k\tilde{P}_t^ku, E^ku)}{t}	=\lim_{t\downarrow0}\frac{(E^ku-P_tE^ku, E^ku)}{t}
								&=\mathcal E(E^ku)
								\\
								&=\mathcal E^k(u),				
							\end{align*}
							where $\mathcal E(E^ku)$ and $\mathcal E^k(u)$ denote $\mathcal E(E^ku,E^ku)$ and $\mathcal E^k(u,u)$, respectively.
						This implies the generator matrices of $\pi^k\circ X$ and $X^k$ are the same. Since both $\pi^k\circ X$ and $X^k$ are minimal processes (i.e., do not come back from the 
						cemetery $\partial$ to $S^k$), by a general theory of Markov chains, we have $\tilde{P}^k_t=P^k_t$. 
						Thus, by \eqref{pMp_k_1}, we obtain $E^kP^k_tf(x)=P_tE^kf(x)$ for any $f \in L^2(S^k;\mu^k)$.
						Thus, we completed the proof of {\bf (i) $\Rightarrow$ (ii)} and the latter assertion of {\bf (1)}.

						{\bf(ii) $\Rightarrow$ (i):} Let $S^k_{\partial}$ and $S_{\partial}$ be the one-point compactification of $S^k$ and $S$, respectively. For any real-valued bounded $\mathcal B(S^k_{\partial})/\mathcal B(\R)$-measurable function $f$ on $S^k_{\partial}$ 
					with $f(\partial)=0$, we have
						\begin{align}
							\mathbb E_x(f(\pi^k\circ X_{t+s}))|\mathcal M_t) 
							=P_sE^kf(X_t)
							=E^kP^k_sf(X_t)
							=P^k_sf(\pi^k \circ X_t). \notag
						\end{align}
					Taking $\tilde{p}_t^k(x,y)=P_t^k\1_y(x)$, we have shown (pMp)$_k$.
						We have completed the proof of {\bf (ii) $\Rightarrow$ (i)}.
									
						{\bf (2):} We show that {\bf (ii)} holds under the assumptions of {\bf (2)}. By the standard argument of Laplace transform, it suffices to show that 
							$E^kG_{\lambda}^k = G_{\lambda} E^k$ on $L^2(S^k;\mu^k)$
						for any $\lambda \ge 0$.

						Note that, for $g\in L^2(S;\mu)$, the image $G_{\lambda}g$ is a unique function $h \in \mathcal F$ such that
								$\mathcal E_{\lambda}(h,v) = \langle g,v \rangle$ $(\forall v \in \mathcal F).$
						Since $D_0$ is a core of $\mathcal F$, to prove (i) it suffices to show that, for $f \in L^2(S^k;\mu^k)$, 
							$\mathcal E_{\lambda}(E^kG^k_{\lambda}f,v) = \langle E^kf,v \rangle $ $(\forall v \in D_0).$
						To show this, we show the following equality:
							\begin{align}
								\mathcal E_{\lambda}(E^kG^k_{\lambda}f,v)=\mathcal E_{\lambda}(E^kG^k_{\lambda}f,E^k\Pi^kv). \label{eq: P=A--1}
							\end{align}
						Since $\Pi^k$ is the adjoint of $E^k$ and $E^k$ is isometric, we have
							\begin{align}
								\langle E^kG^k_{\lambda}f,v\rangle=\langle G_{\lambda}^kf, \Pi^kv\rangle_k=\langle E^kG_{\lambda}^kf,E^k\Pi^kv\rangle. \label{eq: P=A-0}
							\end{align}
						By  (BC)$_k$, we have
							\begin{align}
								&\mathcal E(E^kG^k_{\lambda}f,v) \notag
								\\
								&=\int_{S\times S\setminus d}(E^kG^k_{\lambda}f(x)-E^kG^k_{\lambda}f(y))(v(x)-v(y))J(x,y)\mu(dx)\mu(dy) \notag
								\\
								&= \sum_{\substack{i,j \in S^k \\ i \neq j}} (E^kG^k_{\lambda}f(i)-E^kG^k_{\lambda}f(j))J(i,j)\int_{B_i^k \times B_j^k}(v(x)-v(y))\mu(dx)\mu(dy) \notag
								\\
								&= \sum_{ i \neq j} (E^kG^k_{\lambda}f(i)-E^kG^k_{\lambda}f(j))J(i,j)(\Pi^kv(i)-\Pi^kv(j))\mu^k(i)\mu^k(j)	 \notag	
								\\
								&= \sum_{ i \neq j} (E^kG^k_{\lambda}f(i)-E^kG^k_{\lambda}f(j))J(i,j)\int_{B_i^k \times B_j^k}(E^k\Pi^kv(x)-E^k\Pi^kv(y))\mu(dx)\mu(dy) \notag
								\\
								&=\int_{S\times S\setminus d}(E^kG^k_{\lambda}f(x)-E^kG^k_{\lambda}f(y))(E^k\Pi^kv(x)-E^k\Pi^kv(y))J(x,y)\mu(dx)\mu(dy)\notag
								\\
								&=\mathcal E(E^kG^k_{\lambda}f,E^k\Pi^kv). \notag
						\end{align}
						Thus, we obtain \eqref{eq: P=A--1}.
						Hence, we have
							\begin{align}
						\mathcal E_{\lambda}(E^kG^k_{\lambda}f,v)
								=\mathcal E^k_{\lambda}(G^k_{\lambda}f,\Pi^kv) \notag
								=\langle f,\Pi^kv\rangle_k \notag
								=\langle E^kf, v\rangle.  \notag
							\end{align}
						Here we used (iii) of Proposition \ref{thm: EkEEk}.
						Thus, we have completed the proof of {\bf (2)}.
\end{proof}
	\section{Examples} \label{sec: examples}
		\subsection{Examples of ultrametric spaces}
	We give the several examples of ultrametric spaces included in our setting.
	\begin{example}[$p$-adic ring]\label{ex: p-adic}\normalfont
	Fix an integer $p \ge 2$. Note that $p$ need not be prime.
		Define the {\it$p$-adic ring} $\Q_p$ as follows:
 			\begin{align}
				\Q_p =\{(x_i)_{i \in \Z} \in \{0,1,...,p-1\}^{\Z}: x_i=0 \text{ for all $i \le M$ for some $M$}\}. \notag
			\end{align}
		From the algebraic viewpoint, $\Q_p$ has a natural ring structure by identifying $(x_i)\in \Q_p$ with the formal power series $\sum_{i =-\infty}^{\infty}x_ip^i$.
		However, we do not use any algebraic structures of $\Q_p$ in this paper.
		Let us equip $\Q_p$ with the ultrametric $\rho_p $ defined by
				\begin{align}
					\rho_p(x,y) = p^{-r(x,y)}, \notag
				\end{align}
				where 
					$r(x,y) := \min\{i\in\Z: x_i\neq y_i\},$
				 $\min\emptyset :=\infty$ and $p^{-\infty} :=0$.
				 Let $\mu_p$ be the Haar measure on $\Q_p$ normalized as $\mu(B_0^0)=1$.
			The set $(\Q_p,\rho_p, \mu_p)$ satisfies (U.\ref{cond: U.0})-(U.\ref{cond: U.3}).
	\end{example}
	\begin{example}[leaves of a multibranching tree] \normalfont	\label{ex: leaves}
		We introduce {\em leaves of a multibranching tree}, which is a generalization of $p$-adic rings.
		Let us define 
			\begin{align}
				\mathbb S^{\infty} = \{x=(x_i)_{i \in \Z} \in \N_0^{\Z}: x_i=0 \text{ for all $i \le M$ for some $M$}\}. \notag
			\end{align}
		For a fixed $k \in \Z$, let us define 
			\begin{align}
				\mathbb S^k = \{x=(x_i)_{i\le k}: x_i=0 \text{ for all $i \le M$ for some $M$}\}. \notag
			\end{align}
		Define a map $\{\cdot\}_k: \mathbb S^{\infty} \to \mathbb S^k$ by
			\begin{align}
				 \{x\}_k := (...,x_{k-1},x_k) \quad \text{for $x=(x_i)_{i \in \Z} \in \mathbb S$.} \notag
			\end{align}
		 Let $\mathbb S:=\coprod_{k \in \Z}\mathbb S^k.$
		Let $V$ be an arbitrary function from $\mathbb S$ to $\N$.  We define 
			\begin{align}
				\mathbb S_V =\{(x_i)_{i \in \Z}\in \mathbb S^{\infty} : 0 \le x_i \le V(\{x\}_{i-1}) \text{ for all } i\}. \notag
			\end{align}
		The set $\mathbb S_V$ is called {\em leaves of a multibranching tree}. 
			Let $q>1$ be given and  we metrize $\mathbb S_V$.
		For all $x,y\in \mathbb S_V$, we define a metric $\rho$ as follows:
			\begin{align}
				\rho_q(x,y) = q^{-r(x,y)} \notag
			\end{align}
			where $r(x,y) := \min\{i\in\Z: x_i\neq y_i\}$. 
		If no confusion may occur, we drop the subscript $q$.
		We may define a Radon measure $\mu_V$ on $\mathbb S_V$ such that
			\begin{align}
				\mu_V(B(0,1)) =1\quad \text{and}\quad \mu_V(B_x^k) = V(\{x\}_k)\mu_V(B_x^{k+1}) \quad (\forall x \in \mathbb S_V, \forall k\in\Z). \label{def of Radon on Sv}
			\end{align}
		In Proposition \ref{prop: USS}, we show that $(\mathbb S_V, \rho_q, \mu_V)$ satisfies (U.\ref{cond: U.0})-(U.\ref{cond: U.3}).
		This is a generalization of $(\Q_p,\rho_p,\mu_p)$. Indeed, we can obtain $(\Q_p,\rho_p,\mu_p)$ from $(\mathbb S_V, \rho_q, \mu_V)$ through setting $q = p$ and $V \equiv p-1$.
			\end{example}
			
	Now we show that $(\mathbb S_{V}, \rho_q, \mu_V)$ satisfies (U.\ref{cond: U.0})-(U.\ref{cond: U.3}).
		\begin{prop} \label{prop: USS}
			$(\mathbb S_V, \rho_q, \mu_V)$ satisfies (U.\ref{cond: U.0})-(U.\ref{cond: U.3}).
		\end{prop}
		\begin{proof}
			The condition (U.\ref{cond: U.2}) and (U.\ref{cond: U.3}) clearly hold. Thus, we show (U.\ref{cond: U.0}), (U.\ref{cond: U.1}).
			First, we show (U.\ref{cond: U.1}).
			We note that, for all $x \in \mathbb S_V$ and $k \in \Z$, 
				there exists a finite sequence $\{a_i\}_{0\le i \le p-1} \subset B_x^k$ such that $\{B_{a_i}^{k+1}\}_{0\le i\le p-1}$ is disjoint and 
					\begin{align}
						B_x^k = \sum_{i:\text{ finite}}B_{a_i}^{k+1}. \label{equality: decomposition}
					\end{align}
				See Albeverio-Karwowski \cite[Section 2]{Albeverio:2008uk}.
				We also note that $(\mathbb S_V,\rho_q)$ is complete. See Albeverio-Karwowski \cite[Proposition 2.8]{Albeverio:2008uk}.
		  		By the completeness,  it suffices to see that any closed ball is totally bounded. For any $\epsilon >0$, let $l \in \Z$ such that 
				$0<q^{-l}<\epsilon$.  By using equality (\ref{equality: decomposition}) inductively, we have that  
				there exists a finite sequence $\{b_i\}_{i} \subset B_x^k$ such that $\{B_{b_i}^{l}\}_{i}$ is disjoint and
						$B_x^k = \sum_{i: \text{ finite}}B_{b_i}^{l}.$
				Since $B_{b_i}^l$ is open by (iii) of Fact \ref{fact}, we have checked (U.\ref{cond: U.1}).
				
				Now we show (U.\ref{cond: U.0}). We have already seen as above that $(\mathbb S_V, \rho_q)$ is a complete ultrametric space. The local-compactness is clear by (U.\ref{cond: U.1}).
				Then, (U.\ref{cond: U.0}) is satisfied. We have  completed the proof.
						\end{proof}
						
\subsection{Mixed class}	 \label{subsec: Mixed class}
We introduce a new class of Hunt processes on $\mathbb S_V$. This class is a kind of generalization of the class constructed by Kigami  \cite{KIGAMI:wb}.
We recall the class of Hunt processes on $\mathbb S_V$ constructed in \cite{KIGAMI:wb}.
We only consider the conservative case in \cite{KIGAMI:wb}.
Let $(S,\rho)=(\mathbb S_V, \rho)$ be as in Example \ref{ex: leaves}.
Let $\mathsf B=\{\{x\}_k \in \mathbb S^k: x \in \mathbb S_V, k \in \Z\}$.
Let $\lambda: \mathsf B \to [0,\infty]$  be a function such that 	
	\begin{enumerate}[($\lambda$.1)]
		\setcounter{enumi}{0}
		\item \label{cond: lambda.1} $\dis \sum_{-\infty}^0|\lambda(\{0\}_{m+1}) - \lambda(\{0\}_{m})| < \infty, $ \label{ineq: lambda}
	\end{enumerate}	
and, for all $x,y \in \mathbb S_V$ with $x\neq y$, define 
	\begin{align}
		J(x,y) = J_{\lambda, \mu}(x,y):= \sum_{m=-\infty}^{r(x,y)}\frac{\lambda(\{x\}_{m})-\lambda(\{x\}_{m-1})}{\mu(B^m_x)} \ge 0. \label{ineq: J}
	\end{align} 
For a fixed $x \in \mathbb S_V$, we see that $J(x,\cdot)$ depends only on $r(x,\cdot)$. We sometimes write $J_x^{k}=J(x,y)$ for $y$ such that $r(x,y)=k$.

Let us define $(\mathcal E, \mathcal F)$ as the following symmetric bilinear form:
		\begin{align*}
			\mathcal E (u,v)&= \frac{1}{2}\int_{\mathbb S_V \times \mathbb S_V \setminus d} ( u(x) - u(y) )(v(x)-v(y)) J(x,y) \mu(dx)\mu(dy),  
			\\
			 \mathcal F &= \{f \in L^2(\mathbb S_V;\mu): \mathcal E(f,f) < \infty\}.
		\end{align*}
By \cite[Section 3]{KIGAMI:wb}, $(\mathcal E, \mathcal F)$  is a regular Dirichlet form with core $D_0$. Hence the domain $\mathcal F$ is equal to 
the closure of $D_0$ with respect to $\mathcal E_1$-norm. 
The above Dirichlet form is determined by the following pair $(\lambda, \mu)$:
	\begin{align}
		\Theta_{K,c} =\{(\lambda,\mu) \in  l^+(\mathsf B)\times \mathcal M(\mathbb S_V):\text{ ($\lambda$.\ref{cond: lambda.1}) and (\ref{ineq: J}) hold.}\}, \notag
	\end{align}
	where $\mathcal M(\mathbb S_V)$ stands for the set of Radon measures on $(\mathbb S_V,\rho)$ and  $l^+(\mathsf B)$ stands for the set of functions $\lambda: \mathsf B \to [0,\infty]$.
Note that Dirichlet forms corresponding to elements of $\Theta_{K,c}$ are conservative and the subscription {\it c} means {\it conservative}.
 
	Now we introduce the new class of Hunt processes.
	Let us take the following pairs
		\begin{align}
		(\lambda_1,\mu), (\lambda_2,\mu),...,(\lambda_l,\mu) \in \Theta_{K,c}. \label{asp: mixed1}
		\end{align} 
	We set the jump densities as follows (See \eqref{ineq: J}):
		\begin{align}
			J_1=J_{\lambda_1,\mu}, J_2=J_{\lambda_2,\mu},...,J_l=J_{\lambda_l,\mu}. \notag
		\end{align}
	Let $N_l(\mathbb S_V)$ denote the set of functions $f:\mathbb S_V \times \mathbb S_V \setminus d \to \{1,2,...,l\}$. 
	Let $\Gamma_l  \in N_l(\mathbb S_V)$ such that, 
	 	for each $k \in \Z$ and $i,j \in \mathbb S_V^k$ with $i\neq j$, 
						\begin{align}
							\Gamma_l|_{B_i^k \times B_j^k} = N^k_{ij}, \label{eq: Gamma}
						\end{align}
						where $N^k_{ij} \in \{1,2,...,l\}$ denotes  constant which depends only on $k,i$ and $j$.
	Define the mixed jump density function as follows:
		\begin{align}
			J_{\Gamma_l}(x,y) = J_{\Gamma_l(x,y)}(x,y). \notag
		\end{align}
	Let $(\mathcal E_{\Gamma_l},D_0)$ denote the following symmetric bilinear form:
	\begin{align*}
			&\mathcal E (u,v)= \frac{1}{2}\int_{\mathbb S_V \times \mathbb S_V \setminus d} ( u(x) - u(y) )(v(x)-v(y)) J_{\Gamma_l}(x,y) \mu(dx)\mu(dy) \quad (u,v \in D_0).
		\end{align*}
		\begin{prop}  \label{thm: Mixed DF}
					$(\mathcal E_{\Gamma_l}, D_0)$ satisfies (A.\ref{cond: A.1}), (A.\ref{cond: A.2}) and (BC)$_{\infty}$.
			\end{prop}
		\begin{proof}
			 The condition (A.\ref{cond: A.2}) and (BC)$_{\infty}$ are obvious. 
			We check (A.\ref{cond: A.1}).
			By definition, we have 
				\begin{align}
					J_{\Gamma_l}(x,y) \le \sum_{i=1}^lJ_i(x,y) \quad \text{for all $(x,y) \in \mathbb S_V\times \mathbb S_V\setminus d$.} \notag
				\end{align}
				Note that jump densities corresponding to $\Theta_{K,c}$ satisfy (A.\ref{cond: A.1}) (See \cite[Theorem 3.7]{KIGAMI:wb}). Thus, for all $k \in \Z$ and $i \in \mathbb S_V^k$, we have
				\begin{align}
					\int_{B_i^k\times (B_i^k)^c} J_{\Gamma_l}(x,y)\mu(dx) \mu(dy) \le \sum_{i=1}^l\int_{B_i^k\times (B_i^k)^c} J_i(x,y)\mu(dx) \mu(dy)  < \infty. \notag
				\end{align}
			Hence we have checked (A.\ref{cond: A.1}).
						\end{proof}
				It is easy to see that $(\mathcal E_{\Gamma_l}, \overline{D_0}^{\mathcal E_{\Gamma_l}})$ is a regular Dirichlet form.
				Note that this is determined by $\mu, \lambda_1, \lambda_2,..., \lambda_l$ and $\Gamma_l$.
				Now we define the new class, a generalization of $\Theta_{K,c}$:
		\begin{defn}\normalfont
		The following class is called a {\it mixed class}:
		\begin{align}
		\Theta_{Mix}=\{&(\mu, \lambda_1,\lambda_2,...,\lambda_l, \Gamma_l) \in \mathcal M(\mathbb S_V)\times l^+(\mathsf B)\times \cdot\cdot\cdot \times l^+(\mathsf B)\times N_l(\mathbb S_V) \notag
		\\
		&:\text{ \eqref{asp: mixed1} and \eqref{eq: Gamma} hold.}\}. \notag
	\end{align}
		\end{defn}
	By Proposition \ref{thm: Mixed DF} and Theorem \ref{thm: P=A}, 
the class of Hunt processes associated with the mixed class holds the projection Markov property of any level.  By Corollary \ref{cor}, the class of Hunt processes associated with the mixed class can be approximated almost surely by the Markov chains associated with the averaged forms, which are identical in law with the projected processes. 
	
				
\section*{Acknowledgements}
\addcontentsline{toc}{section}{a}
\markright{a}
I would like to thank Kouji Yano for valuable discussions and very careful reading of earlier versions of this paper. I also would like to thank Hiroshi Kaneko,  Takashi Kumagai and Masanori Hino for valuable discussions and comments.


\end{document}